\documentclass[a4paper,11pt,reqno]{amsart}
\usepackage{amsmath}
\usepackage[T1]{fontenc}
\usepackage{amssymb}
\usepackage{amsthm}
\usepackage{color}
\usepackage[lmargin=2.5 cm,rmargin=2.5 cm,tmargin=3.5cm,bmargin=2.5cm,
paper=a4paper]{geometry}

\newcommand{\Ab}{\mathbf A}

\newcommand{\Fb}{\mathbf F}

\newcommand{\R}{\mathbb R}
\newcommand{\C}{\mathbb C}
\newcommand{\E}{\mathrm{E}_{\rm gs}(\kappa, H)}

\DeclareMathOperator{\curl}{curl}\DeclareMathOperator{\Div}{div}

\newtheorem{thm}{Theorem}[section]
\newtheorem{prop}[thm]{Proposition}
\newtheorem{lem}[thm]{Lemma}

\theoremstyle{remark}

\newtheorem{rem}[thm]{Remark}
\newtheorem{defn}[thm]{Definition}

\numberwithin{equation}{section}

\title[Ginzburg-Landau density]
{The density of superconductivity in the bulk regime}

\author{Bernard Helffer}
\author{Ayman Kachmar}
\address[B. Helffer]{Laboratoire Jean Leray, Universit\'e de Nantes, 2 rue de la Houssini\`ere, 44322 Nantes (France) and Laboratoire de Math\'ematiques, Univ. Paris-Sud. }
\email{bernard.helffer@univ-nantes.fr}
\address[A. Kachmar]{Department of Mathematics, Lebanese University, Nabatieh, Lebanon.}
\email{ayman.kashmar@gmail.com}
\date{\today}

\begin{document}

\begin{abstract}
In the asymptotic limit of a large Ginzburg-Landau parameter, we give a new asymptotic formula for the $L^2$-norm of the Ginzburg-Landau order parameter. The formula is  valid in the bulk regime where the intensity of the applied magnetic field is  of the same order as the Ginzburg-Landau  parameter and strictly below the second critical field. Our formula complements the celebrated  one of Sandier-Serfaty for the $L^4$-norm. 
\end{abstract}

\maketitle 

\section{Introduction and main results}\label{hc2-sec:int}

\subsection*{The Ginzburg-Landau model}

The Ginzburg-Landau functional  is defined
as the sum of two functionals, the  {\it energy of the order parameter} and the {\it  magnetic energy}. It reads  as follows,
\begin{equation}\label{eq-3D-GLf}
\mathcal E_{\rm GL}(\psi,\Ab)=\mathcal E_{\rm op}(\psi,\Ab)+\mathcal E_{\rm mag}(\Ab)\,,
\end{equation}
where
\begin{equation}\label{eq:op-mag}
\begin{aligned}
\mathcal E_{\rm op}(\psi,\Ab)&=
\int_\Omega \left(|(\nabla-i\kappa H\Ab)\psi|^2-\kappa^2|\psi|^2+\frac{\kappa^2}2|\psi|^4\right)\,dx\,,\\
\mathcal E_{\rm mag}(\Ab)&=\kappa^2H^2\int_\Omega|\curl\Ab-1|^2\,dx\,.
\end{aligned}
\end{equation}
Here:
\begin{itemize}
\item $\Omega\subset\R^2$ is an open, bounded and simply connected
set with a $C^\infty$  boundary\,; $\Omega$ is the cross section of a
cylindrical superconducting sample placed vertically.
\item $(\psi,\Ab)\in H^1(\Omega;\mathbb C)\times H^1(\Omega;\mathbb
R^2)$ describes the state of superconductivity as follows:
$|\psi|^2$ measures the local density of the superconducting Cooper
pairs and $\curl\Ab$ measures the induced magnetic field in the
sample.
\item $\kappa>0$ is the Ginzburg-Landau parameter, a material characteristic
of the sample.
\item $H>0$ measures the intensity of the applied magnetic
field.
\item  The
applied magnetic field is $\kappa H \vec{e}$, where $\vec{
e}=(0,0,1)$.
\end{itemize}
We introduce the ground state energy of the functional in
\eqref{eq-3D-GLf}:
\begin{equation}\label{eq-gse}
\E=\inf\{\mathcal E_{\rm GL}(\psi,\Ab)~:~(\psi,\Ab)\in H^1(\Omega;\mathbb
C)\times H^1(\Omega;\mathbb R^2)\}\,.
\end{equation}
For a given $(\kappa,H)$, a configuration $(\psi,\Ab) \in H^1(\Omega;\mathbb
C)\times H^1(\Omega;\mathbb R^2)$ satisfying \break  $\mathcal E_{\rm GL}(\psi,\Ab)=\E$  is called  a minimizer of the functional $\mathcal E_{\rm GL}$ and we will denote it by $(\psi,\Ab)_{\kappa,H}$ to emphasize its dependence on $\kappa$ and $H$. Such a minimizer is a solution of the following Ginzburg-Landau equations (we use the notation $\nabla^\bot=(\partial_{x_2},-\partial_{x_1})$)
\begin{equation}\label{eq:GL}
\left\{
\begin{array}{rll}
-\big(\nabla-i\kappa H\Ab\big)^2\psi&=\kappa^2(1-|\psi|^2)\psi &{\rm in}\ \Omega\,,\\
-\nabla^{\perp}  \curl\Ab&= (\kappa H)^{-1}{\rm Im}\big(\overline{\psi}\,(\nabla-i\kappa H {\bf A})\psi\big) & {\rm in}\ \Omega\,,\\
\nu\cdot(\nabla-i\kappa H\Ab)\psi&=0 & {\rm on}\ \partial \Omega \,,\\
\curl \Ab &=B_0 & {\rm on}\ \partial \Omega \,.
\end{array}
\right.
\end{equation} 
\subsection*{Gauge invariant quantities}
The  physically relevant quantities, {\it density}, {\it induced magnetic field}, {\it energy} and {\it supercurrent} are  invariant under the Gauge transformations. More precisely, the following quantities 
\begin{align}
&|\psi|^2\,,\quad\curl\Ab\,,\quad |(\nabla-i\kappa H\Ab)\psi|^2\,,\label{eq:den-mf-en}\\
&j(\psi,\Ab)={\rm Re}\big(-i\overline{\psi}\,(\nabla-i\kappa H\Ab)\psi\big) \label{eq:sc}\,,
\end{align}
are invariant under the transformation $(\psi,\Ab)\mapsto (e^{i\chi},\Ab-\nabla\chi)$ for every given $\chi\in H^1(\Omega;\R)$. This gauge invariance insures that all the quantities in \eqref{eq:den-mf-en} and \eqref{eq:sc} are smooth functions (cf. \cite[Ch.~2]{SS-b})  when $(\psi,\Ab)$ is a minimizer.  The solution $(\psi,\Ab)$ of \eqref{eq:GL}  in the class such that  ${\rm div} \Ab =0$  in $\Omega$ and $\Ab \cdot \nu =0$ on $\partial \Omega$ is indeed $C^\infty$. So, without loss of generality, when working with a  solution $(\psi,\Ab)$  of \eqref{eq:GL}, will assume that  it is $C^\infty$.

\subsection*{Earlier results on the density}

In this paper, we will study the asymptotics for the density  in the following regime
\begin{equation}\label{eq:regime}
H=b\kappa\,,
\end{equation}
where $ b\in(0,1)$ is a {\bf fixed} constant. 

This corresponds to the situation of an external magnetic field with intensity strictly below   the second critical field $H_{c_2} (\kappa):=\kappa\,$.  The case where $b>1$ in \eqref{eq:regime} is related to the phenomenon of surface superconductivity which is extensively studied by many authors \cite{Al-s, CR, FH-b,   P-cmp}.

When \eqref{eq:regime} holds, Sandier-Serfaty \cite{SS02} proved the following formula for the ground state energy in \eqref{eq-gse}:
\begin{equation}\label{eq:en}
\E= g(b)|\Omega|\kappa^2+o(\kappa^2)\quad{\rm as}~\kappa\to +\infty\,,
\end{equation}
where $g(b)$ is an implicitly defined quantity that depends only on $b$. Its precise definition will be  given in \eqref{eq:g}. In particular, it satisfies:
$$g(0)=-\frac12\,, \, g(1)=0 \mbox{  and } g(b)<0 \mbox{  for } b\in(0,1)\,.
$$

The convergence in \eqref{eq:en} is uniform with respect to $b$ on every interval $[\epsilon,1)$, $\epsilon>0$. The uniform convergence fails on the interval $(0,1)\,$.  More details regarding the uniformity with respect to $b$ are given by K. Attar  in \cite{Att, Att2}.

Now suppose that \eqref{eq:regime} holds and that $(\psi,\Ab)_{\kappa,H}$ is a minimizer of the functional in \eqref{eq-3D-GLf}. The magnetic energy satisfies \cite{Att}:
\begin{equation}\label{eq:m-e}
\kappa^2H^2\int_\Omega|\curl\Ab-1|^2\,dx\leq C\, \kappa^{7/4}\,,
\end{equation}
for $\kappa\geq \kappa_0$, where $\kappa_0$ and $C$ are two constants  that depend only on the domain $\Omega$ and the constant $b$ in \eqref{eq:regime}.  Hence its contribution in the  ground state energy  is relatively small as $\kappa \rightarrow +\infty\,$.

Again, if $b\in[\epsilon,1)$ for some $\epsilon>0\,$, the constants $\kappa_0$ and $C$ can be selected independently from $b$, but they will depend on $\epsilon\,$. More details can be found in \cite{Att, Att2}, where it is allowed for $\epsilon$ to depend on $\kappa$, $\epsilon =\epsilon (\kappa)$, and approach $0$ as $\kappa\to+\infty\,$. 

Using the Ginzburg-Landau equation for $\psi$ (see \eqref{eq:GL}), we get the following simple relation between the energy and the $L^2$-norm of the density:
\begin{equation}\label{eq:L4=en}
\mathcal E_{\rm op}(\psi,\Ab)=-\frac{\kappa^2}2\int_\Omega|\psi (x) |^4\,dx\,,
\end{equation}
where $\mathcal E_{\rm op}$ is the energy of the order parameter introduced in \eqref{eq:op-mag}. Consequently, combining the estimates in \eqref{eq:en} and \eqref{eq:m-e}, we deduce the following formula regarding the  $L^2$-norm of the density \cite{SS02}:
\begin{equation}\label{eq:L4}
\int_\Omega|\psi (x) |^4\,dx=-2g(b)|\Omega|+o(1)\quad{\rm as~}\kappa\to + \infty\,,
\end{equation}
where the function $o(1)$ is dominated by a function $s(\kappa)$ such that $s(\kappa)$ is independent of the choice of the minimizer $(\psi,\Ab)_{\kappa,H}$ and $s(\kappa)\to 0$ as $\kappa\to + \infty\,$.  When $b\in[\epsilon,1)$ for some $\epsilon>0\,$, the function $s(\kappa)$ can be selected independently from $b$.  More details can be found in \cite{Att, Att2}, where the case $\epsilon =\epsilon (\kappa)$ tending to $0$ is considered. In particular the comparison of $\epsilon (\kappa)$ with the first critical field $H_{c_1}(\kappa)\approx \frac{ \ln\kappa}\kappa$ could play a role. 

Furthermore, Sandier-Serfaty obtained  the following weak-convergence of $|\psi|^4$ as $\kappa\to +\infty$ in the sense of distributions \cite{SS02}:
\begin{equation}\label{eq:dist-L4}
|\psi|^4\rightharpoonup -2g(b)\quad{\rm in~}\mathcal D'(\Omega)\,.
\end{equation}
\subsection*{Open questions}
 
Note that for $b=0$ in \eqref{eq:regime}, i.e. $H=0$, every minimizer $(\psi,\Ab)_{\kappa,H}$ satisfies $|\psi|=1$ and $\curl\Ab=1$. This is consistent with \eqref{eq:L4} and \eqref{eq:m-e}. Indeed, as $b\to0_+$, we know that $g(b)\to-\frac12$. 

The regime $b\to0_+$ (which corresponds to $H\ll\kappa$, see \eqref{eq:regime}) is thoroughly analyzed by Sandier-Serfaty in \cite{SS-rmp, SS-cmp}. In particular, it is proved that, for any minimizer $(\psi,\Ab)_{\kappa,H}$,   the density $|\psi|^2$ satisfies $|\psi|^2\to 1$ in $L^2(\Omega)$  and it  is close to $1$ everywhere except in narrow regions of area $O(\kappa^{-1})$. The region where $|\psi|^2$ is not close to $1$ consists of small defects accommodating isolated zeros of $\psi$, called {\it vortices}. These vortices are evenly distributed in the domain $\Omega$ along a lattice, and the distance between two vortices is $\approx H^{-1}$,  much larger than $\kappa^{-1}$, the core size of the vortex.

The detailed analysis of the distribution of  vortices 
is missing when \eqref{eq:regime} holds for a fixed constant $b\in(0,1)$, even for small values of $b$. This is a challenging problem mainly for the following reason. For a minimizer $(\psi,\Ab)_{\kappa,H}$, it is expected that $\psi$ will have isolated zeros/vortices filling up all the domain $\Omega$, but these zeros are separated by a distance $O(H^{-1})=O(\kappa^{-1})$. At the same time, the core-size of every vortex is equal to $O(\kappa^{-1})$. Consequently, detecting the vortices in this regime becomes harder than  when $H\ll\kappa$ (i.e. $b\to0_+$ in \eqref{eq:regime}).

This problem is related to the one of the Abrikosov state near the critical field $H_{C_2}:=\kappa\,$, where the transition to the normal state in the bulk occurs. This is visualized in  the regime $b\to1_-$ in \eqref{eq:regime} and is analyzed in many papers, \cite{Al, FK-am, Ka-SIMA, KN}. The same difficulty is encountered when trying to detect the vortices by the methods of Sandier-Serfaty, so that the analysis  is shifted to the distribution of the density $|\psi|^2$ instead.

In this paper, we  complement the results of Sandier-Serfaty by obtaining analogues of the formulas in \eqref{eq:L4} and \eqref{eq:dist-L4} for the density $|\psi|^2$ (instead of the square of the density, $|\psi|^4$), in the regime where \eqref{eq:regime} holds for a fixed constant $b\in(0,1)$. Besides that such results are new and do not follow from the analysis by Sandier-Serfaty \cite{SS02}, they  might be helpful in the analysis of the vortices. Related to these results is the asymptotics of the supercurrent
$j(\psi,\Ab)$ when \eqref{eq:regime} holds. Even in the particular regime $H\ll \kappa$ (i.e. $ b \ll 1$ in \eqref{eq:regime}),  the analysis of the distribution of the super-current is missing. Actually,   Sandier-Serfaty \cite[Ch.~8, Corol.~8.1]{SS-b} prove only that, in the regime $\frac{|\ln\kappa|}\kappa\ll H \ll \kappa$,
$\curl j\to 0$ ${\rm in~}\mathcal D'(\Omega)$,
 as $\kappa\to +\infty$.



\subsection*{Main results}

To state our main results, we recall  some  properties of $g$. The function $g$ is increasing and  concave (cf. \cite[Thm.~2.1]{FK-cpde}).  Consequently, $g$ has at each point left- and right-sided derivatives $g'(b_-)$ and $g'(b_+)$ with
$$
g'(b_+) \leq g'(b_-) \,.
$$ 
Therefore, we can  introduce the set  
\begin{equation}\label{eq:reg}
\mathcal R=\{b\in(0,1)~:~g'(b_-)=g'(b_+)\}
\end{equation}
whose complement  in the interval $(0,1)$ is countable.
%
Assuming that $b\in\mathcal R$  and \eqref{eq:regime} holds, we will prove that every minimizer $(\psi,\Ab)_{\kappa,H}$ of the G-L functional in \eqref{eq-3D-GLf} satisfies (compare with \eqref{eq:L4})
\begin{equation}\label{eq:L2}
\int_\Omega|\psi (x)|^2\,dx=\Big(bg'(b)-2g(b)\Big)|\Omega|+o(1)\quad{\rm as~}\kappa\to +\infty\,.
\end{equation}
The formula in \eqref{eq:L2} is consistent with the one given in \cite[Eq.~(1.6)]{Ka-SIMA}  which is valid as $b\to1_-$.  We have indeed (see below  \eqref{eq:Ab}), 
 $$g(b)\sim E_{\rm Ab}(b-1)^2 \,,$$
  where $E_{\rm Ab}\in[-\frac12,0)$ is a universal constant.
  
  
More precisely, our main result is:
\begin{thm}\label{thm:HK}

Let $b\in(0,1)$. There exist $\kappa_0>0$ and  a function $\lambda:\R_+\to\R_+$ such that $\displaystyle\lim_{\kappa\to\infty}\lambda(\kappa)=0$ and the following is true.

If $(\psi,\Ab)_{\kappa, H}$ is a minimizer of the functional in \eqref{eq-3D-GLf} for $H=b\kappa$ and $\kappa\geq \kappa_0$, then
\begin{enumerate}
\item 
$$bg'(b_+)-\lambda(\kappa)\leq \frac1{\kappa^2|\Omega|}\int_\Omega|(\nabla-i\kappa H)\Ab)|^2\,dx\leq bg'(b_-)+\lambda(\kappa)\,.$$
\item 
$$bg'(b_+)-2g(b)-\lambda(\kappa)\leq \frac1{|\Omega|}\int_\Omega |\psi(x)|^2\,dx\leq bg'(b_-)-2g(b)+\lambda(\kappa)\quad{\rm as~}\kappa\to +\infty\,.$$
\item If $b\in\mathcal R$, then as $\kappa\to\infty$, the following convergence holds in the sense of distributions
$$|\psi|^2\rightharpoonup bg'(b)-2g(b)\quad{\rm in }~\mathcal D'(\Omega)\,.$$
\item The supercurrent satisfies
$$\frac1{\kappa^2|\Omega|}\int_\Omega |j(\psi,\Ab)|^2\,dx\leq bg'(b_-)+\lambda(\kappa)\,,$$
and
$$\frac1{\kappa|\Omega|}\int_\Omega |j(\psi,\Ab)|\,dx\leq \sqrt{ bg'(b_-)\big(bg'(b_-)-2g(b)\big)}+\lambda(\kappa)\,.$$
\end{enumerate}
\end{thm}
\begin{rem}\label{rem:LOT}{\bf[On the leading order term]}~

The coefficient of the leading term in \eqref{eq:L2} does not vanish. Actually, $g'(b)\geq 0$ since $g$ is increasing, and  $g(b)<0$ for $b\in(0,1)$. 
\end{rem}

\begin{rem}\label{rem:bto0}{\bf[On the $L^2$-norm of $1-|\psi|^2$]}~

Using \eqref{eq:L4} and H\"older's inequality, we get, for fixed $b$ and as $\kappa \rightarrow +\infty$, 
$$
 \frac1{|\Omega|}\int_\Omega |\psi(x)|^2\,dx \leq  |\Omega|^{-\frac 12 } \left( \int_\Omega |\psi(x)|^4\,dx\right)^\frac 12 \leq  (-2 g(b))^{\frac 12}  + o(1)\,.
 $$
 Combined with the lower bound in \eqref{eq:L2}, we get  (we use that $g'(b_+)\geq 0$)
 $$-2g(b)-o(1)\leq \frac1{|\Omega|}\int_\Omega |\psi(x) |^2\,dx\leq (-2 g(b))^{\frac 12}  + o(1)\,.$$
 Now we find the following estimate for the $L^2$-norm of $1-|\psi|^2$,
 $$\frac1{|\Omega|}\int_\Omega(1-|\psi(x)|^2)^2\,dx\leq  1+2g(b)+o(1)\,,$$
 with the principal term on the right hand side approaching $0$ as $b\to0_+$, since $$\lim_{b\to0_+}g(b)=-\frac12\,.$$
 This is consistent with the behavior $|\psi|^2\to 1$ in $L^2(\Omega)$ obtained in \cite{SS-rmp}.
\end{rem}

\begin{rem}\label{rem:pe}{\bf[On the potential energy]}

When $b\in\mathcal R$ (see \eqref{eq:reg}), we get from Theorem~\ref{thm:HK} that the potential energy satisfies
$$
\kappa ^2 \int_\Omega\left(-|\psi(x)|^2+\frac{1}2|\psi(x)|^4\right)\,dx= \kappa^2 \Big(g(b)-bg'(b)\Big)|\Omega|(1+o(1))\,.$$
\end{rem}

\section{Preliminaries}

\subsection{The bulk energy}

Here we give the definition of the reference bulk energy $g(\cdot)$. This energy first appeared in \cite{SS02} and was then extensively studied in \cite{AfSe, FK-cpde, Att2, Att3, Ka-nf}. 

Consider $b\in\,(0,+\infty)$, $r>0\,$ and
$Q_r=\,(-r/2,r/2)\,\times\,(-r/2,r/2)$\,. Define the functional,
\begin{equation}\label{eq:rGL}
F_{b,Q_r}(u)=\int_{Q_r}\left(b\, |(\nabla-i\Ab_0)u|^2-|u|^2+\frac12|u|^4\right)\,dx\,,
\quad \mbox{ for } u\in H^1(Q_r)\,.
\end{equation}
Here, $\Ab_0$ is the magnetic potential,
\begin{equation}\label{eq:A0}
\Ab_0(x)=\frac12(-x_2,x_1)\,,\quad \mbox{ for } x=(x_1,x_2)\in \R^2\,.
\end{equation}
Define the two   ground state energies,
\begin{equation}\label{eq:eD}
\begin{aligned}
&e_D(b,r)=\inf\{F_{b,Q_r}(u)~:~u\in H^1_0(Q_r)\}\,,\\
&e_N(b,r)=\inf\{F_{b,Q_r}(u)~:~u\in H^1(Q_r)\}\,.
\end{aligned}
\end{equation}
The function $g(\cdot)$ may be defined as
follows (cf. \cite{FK-cpde, SS02, Att2}), 
\begin{equation}\label{eq:g}
\forall~b>0\,,\quad g(b)=\lim_{r\to +\infty}\frac{e_D(b,r)}{|Q_r|}=\lim_{r\to\infty}\frac{e_N(b,r)}{|Q_r|}\,,
\end{equation}
where  $|Q_r|$ denotes the area of $Q_r$ ($|Q_r|=r^2$). Furthermore, there exists a  constant $C$ such that, for all
$r\geq 1$ and  $b\in(0,1)$,
\begin{equation}\label{eq:g'}
  g(b)\leq\frac{e_D(b,r)}{|Q_r|}\leq g(b)+C\frac{\sqrt{b}}{r}\quad {\rm and}\quad 
  e_D(b,R)-Cr\sqrt{b}\leq e_N(b,r)\leq e_D(b,r)\,.
\end{equation}
Various properties satisfied by the function $g(\cdot)$ are established in \cite{Att3, FK-cpde,  Ka-JFA, SS02}. In particular,  
the
function $g(\cdot)$ is a monotone non decreasing continuous and locally Lipschitz function such
that
\begin{equation} \label{propg}
g(0)=-\frac12 \mbox{ and } g(b)=0 \mbox{ when } b\geq 1\,,
\end{equation}
and 
\begin{equation}\label{eq:Ab}
\lim_{b\to1_-}\frac{g(b)}{(b-1)^2}=E_{\rm Ab}\in [-\frac12,0)\,.
\end{equation}
\subsection{A priori estimates and Gauge tranformations}

Here we collect useful estimates regarding the critical points of the Ginzburg-Landau functional 
(cf. \cite[Prop.~10.3.1~and~11.4.4]{FH-b}).

\begin{prop}\label{prop:FH-b}
Let $b\in(0,1)$.  There exist two constants
$C>0$ and $\kappa_0>0$ such that, if $\kappa\geq \kappa_0$, $H=b\kappa$ and $(\psi,\Ab)_{\kappa,H}$ is a critical point of \eqref{eq-gse}, then:
\begin{align}
&\|\psi\|_\infty\leq1\,,\label{eq:psi<1}\\
&\|(\nabla-i\kappa H\Ab)\psi\|_{C(\overline\Omega)}\leq C\kappa\,,\label{eq:grad-psi}\\
& \|\curl\Ab-1\|_{C^1(\overline\Omega)}\leq \frac{C}{\kappa}\,.\label{eq:curl=cst}
\end{align}
\end{prop}

%

As a consequence of Proposition~\ref{prop:FH-b}, we may pick a useful gauge transformation in every ball with small radius:

\begin{prop}\label{prop:gauge}
Let $b\in(0,1)$.
There exist two constants $C>0$ and $\kappa_0>0$ such that, for any  $x_0\in\Omega$,
there exists a function $\varphi_0\in C^1(\Omega)$ such that, if $\kappa\geq \kappa_0$, $H=b\kappa$ and $(\psi,\Ab)_{\kappa,H}$ is a $C^\infty$ solution of \eqref{eq:GL}, then:
$$\forall~x\in \Omega\,,\quad
\Big|\Ab(x)-\big(\Ab_0(x-x_0)-\nabla\varphi_0(x)\big)\Big|\leq \frac{C}{\kappa} \max\Big( |x-x_0|,|x-x_0|^2\Big)\,,$$
where $\Ab_0$ is the vector field introduced in \eqref{eq:A0}.
\end{prop} 
\begin{proof}
Let $B=\curl\Ab$.  Choose a convex and open set $U\subset\R^2$ such that $\overline{\Omega}\subset U$. We may extend the function $B$ to a function $B_{\rm ext}:U\to\R$ such that 
\begin{equation}\label{eq:ext-B}
{\rm supp}(B_{\rm ext}) \subset U\quad{\rm  and}\quad \|\nabla B_{\rm ext}\|_{L^\infty(U)}\leq C\, \|\nabla B\|_{L^\infty(\Omega)}\,,
\end{equation} 
where $C$ is a constant that depends solely on $\Omega$ and $U$ (i.e. it is independent of $B$).

Define the vector field in $\Omega$
$$\mathbf G(x)=2\left(\int_0^1s B_{\rm ext}\big(s(x-x_0)+x_0\big)\,ds\right)\Ab_0(x-x_0)\,.$$
It is easy to check that
$$\curl\mathbf  G=B_{\rm ext}=B\quad{\rm in ~}\Omega\,.$$ Consequently, since $\Omega$ is simply connected,  there exists a smooth function $\varphi_0$ such that,
$$\Ab(x)=\mathbf G(x)-\nabla\varphi_0(x)\,.$$
 Using \eqref{eq:curl=cst}, \eqref{eq:ext-B} and the mean value theorem, we get further
$$|\mathbf G(x)-B(x_0)\Ab_0(x-x_0)|\leq \frac{C}{ \kappa}|x-x_0|^2\,.$$
Again, using \eqref{eq:curl=cst}, we write $\Big|(B(x_0)-1)\Ab_0(x-x_0)\Big|\leq C\kappa^{-1}|x-x_0|$. This yields the inequality
$$|\mathbf G(x)-\Ab_0(x-x_0)|\leq \frac{C}{ \kappa}\max\Big(|x-x_0|,|x-x_0|^2\Big)\,.$$

\end{proof}

\begin{rem}\label{rem:gauge}
We will use the inequality in Proposition~\ref{prop:gauge} for $|x-x_0|\leq \ell$ and $\ell\ll1$, which in turn reads as follows
$$\Big|\Ab(x)-\big(\Ab_0(x-x_0)-\nabla\varphi_0(x)\big)\Big|\leq \frac{C}{\kappa}\ell\,.$$ 
\end{rem}



\section{On the local energy of minimizers}\label{sec:ub-loc}

For  any open set $D\subset\Omega$, we define the following local energy
\begin{equation}\label{eq:loc-en}
\mathcal E_0(f,a;D)=\int_D\left(|\nabla-i\kappa Ha)f|^2-\kappa^2|f|^2+\frac{\kappa^2}2|f|^4\right)\,dx\,.
\end{equation}
 For $x_0\in\R^2$ and $\ell>0$, $Q_\ell(x_0)=x_0+(-\ell/2,\ell/2)^2$ denotes  the square of center $x_0$ and side-length~$\ell$.

We will need the following result,  essentially proved in \cite{Att} modulo a few adjustments.

\begin{prop}\label{prop:Ka-SIMA}
If $b\in(0,1)$, there exist positive constants $C\,$, $R_0\,$, and $\kappa_0>0\,,$ such that for $\kappa\geq\kappa_0\,$, $H=b\kappa\,$, $R_0\kappa^{-1}\leq \ell\leq \kappa_0^{-1}\,,$
 $x_0\in\Omega\,$, and if $\overline{Q_\ell(x_0)}\subset\Omega\,$, then the following inequalities hold
$$\left|\frac1{|Q_\ell(x_0)|}\mathcal E_0\Big( e^{i\kappa H\varphi_0}\psi,\Ab_0^{x_0};Q_\ell(x_0)\Big)- 
\kappa^2 g(b)\right|\leq C\Big(\ell+(\kappa\ell)^{-1}\Big)\kappa^2\,,$$
and
$$\left|\frac1{|Q_\ell(x_0)|}\int_{Q_\ell(x_0)}|\psi (x) |^4\,dx + 2g(b)\right|\leq C\Big(\ell+(\kappa\ell)^{-1}\Big)\,,$$
where $\Ab_0^{x_0}(x)=\Ab_0(x-x_0)$, $\Ab_0$ is the vector field in \eqref{eq:A0}, and $\varphi_0$ is the function constructed in Proposition~\ref{prop:gauge}\,.
\end{prop}
\begin{proof}
In \cite[Prop.~4.2 and 6.2]{Att}, it is proved that
\begin{equation}\label{eq:loc-en-Att*}
\left|\frac1{|Q_\ell(x_0)|}\mathcal E_0\Big(\psi,\Ab;Q_\ell(x_0)\Big)- 
\kappa^2 g(b)\right|\leq C\Big(\ell+(\ell\kappa)^{-1}\Big)\kappa^2\,.
\end{equation}
The estimate of the remainder term in \cite{Att} was worse because  the magnetic field was assumed non-constant
and a variant of the inequality in Proposition~\ref{prop:gauge} was used (with a worse error as well). However, in our case of a constant magnetic field, we insert the inequality in Proposition~\ref{prop:gauge} into the proof given in \cite{Att} and get the better remainder
as in \eqref{eq:loc-en-Att*}.

 We write
\begin{align*}
\mathcal E_0\Big(\psi,\Ab;Q_\ell(x_0)\Big)&=
\mathcal E_0\Big(\psi,\Ab_0^{x_0}-\nabla\varphi_0+(\Ab-\Ab_0^{x_0}+\nabla \varphi_0);Q_\ell(x_0)\Big)\\
&\geq (1-\ell)\, \mathcal E_0\Big(\psi,\Ab_0^{x_0}-\nabla\varphi_0;Q_\ell(x_0)\Big)\\
&\quad-\ell^{-1}\kappa^2H^2\int_{Q_\ell(x_0)}|\Ab-\Ab_0^{x_0}+\nabla \varphi_0|^2|\psi|^2\,dx-\ell\kappa^2\int_{Q_\ell(x_0)}|\psi|^2\,dx\,.
\end{align*}  
Using the gauge invariance, the bound $|\psi|\leq 1$ and the inequality in Proposition~\ref{prop:gauge}\,, we get the following lower bound
$$\mathcal E_0\Big(\psi,\Ab;Q_\ell(x_0)\Big)\geq (1-\ell)\, \mathcal E_0\Big(e^{-\kappa H\varphi_0}\psi,\Ab_0^{x_0};Q_\ell(x_0)\Big)-C\kappa^2\ell^3\,.
$$
In a similar fashion, we prove the upper bound
$$\mathcal E_0\Big(\psi,\Ab;Q_\ell(x_0)\Big)\leq (1+\ell)\, \mathcal E_0\Big(e^{-\kappa H\varphi_0}\psi,\Ab_0^{x_0};Q_\ell(x_0)\Big)+C\kappa^2\ell^3\,.$$
Inserting the foregoing lower and upper bounds into \eqref{eq:loc-en-Att*}, we get the first inequality in Proposition~\ref{prop:Ka-SIMA}.

Now we prove the second inequality in Proposition~\ref{prop:Ka-SIMA}. We multiply the first G-L equation in \eqref{eq:GL} by $\overline\psi$ and integrate by parts in the integral over $Q_\ell(x_0)$. We get
$$-\frac{\kappa^2}2\int_{Q_\ell(x_0)}|\psi(x)|^4\,dx=\mathcal E_0\Big(\psi,\Ab;Q_\ell(x_0)\Big)
+\int_{\partial Q_\ell(x_0)}\, \overline{\psi}\;(\nu \cdot (\nabla-i\kappa H\Ab)\psi)\,d\sigma(x)\,.
$$
Using the bounds $|\psi|\leq 1$ and $|(\nabla-i\kappa H\Ab)\psi|\leq C\kappa$ in Proposition~\ref{prop:FH-b}, we get that the boundary term is bounded by  $\tilde C\kappa\ell$, where $\tilde C$ is a constant.

Now, using \eqref{eq:loc-en-Att*}, we get 
$$\left|-\frac{\kappa^2}2\int_{Q_\ell(x_0)}|\psi(x)|^4\,dx-g(b)\kappa^2|Q_\ell(x_0)|\right|\leq 
C\,\Big(\ell+(\kappa\ell)^{-1}\Big)\kappa^2|Q_\ell(x_0)|\,.$$
\end{proof}


\section{Proof of Theorem~\ref{thm:HK}}\label{sec:proof}

Our proof of Theorem~\ref{thm:HK} has some similarities with the analysis of diamagnetism \cite{FH-aif} and the computation of the quantum supercurrent  \cite{F-sc}.

For the proof of Theorem~\ref{thm:HK}, it is easier to work with rescaled  variables.

\begin{defn}\label{def:conv}

Let $x_0\in\Omega$, $\ell>0$ and $f\in H^1(\Omega)$ and suppose that
$Q_{\ell}(x_0)\subset\Omega\,.$
We define the new function $\tilde f$ on $Q_{\ell\sqrt{\kappa H}}:=Q_{\ell\sqrt{\kappa H}}(0)$ as follows:
$$\tilde f(y)=f\left(x_0+\frac{y}{\sqrt{\kappa H}}\right)\,.$$
\end{defn}

For $H=b\kappa$ and $R=\ell\sqrt{\kappa H}$, we have the following relation:
\begin{equation}\label{rem:blow-up}
\frac1{\kappa^2|Q_\ell(x_0)|}\mathcal E_0\big(f,\Ab_0^{x_0};Q_\ell(x_0)\big)=\frac{1}{|Q_R|}\int_{Q_R}
\left(b|(\nabla-i\Ab_0)\tilde f|^2-|\tilde f|^2+\frac12|\tilde f|^4\right)\,dy\,.
\end{equation}

\begin{lem}\label{lem:mag-grad}
For $b\in(0,1)$, there exist $\kappa_0,R_0>0$ and a positive-valued function $\mathrm r(\cdot,\cdot)$ such that $\lim_{(t^{-1},s)\to0}\mathrm r(t,s)=0$ and the inequality 
$$g'(b_+) -\mathrm r(R,\ell)\leq \frac1{|Q_R|}\int_{Q_R}|(\nabla-i\Ab_0) \tilde f|^2\,dy
\leq g'(b_-)+ \mathrm r(R,\ell)\,, $$
holds for (cf. Prop.~\ref{prop:Ka-SIMA})
$$f(x)=e^{i\kappa H\varphi_0}\psi(x)\,,$$
$R=\ell\sqrt{\kappa H}$, $R_0\kappa^{-1}<\ell<\kappa_0^{-1}$, $\kappa\geq \kappa_0$, $H=b\kappa$ and $(\psi,\Ab)_{\kappa,H}$ is a minimizer of  the functional in \eqref{eq-3D-GLf}.
\end{lem}
\begin{proof}
Recall the definition of the function $F_{b,Q_R}$ in \eqref{eq:rGL}. By \eqref{rem:blow-up} and Proposition~\ref{prop:Ka-SIMA}, 
$$F_{b,Q_R}(\tilde f)\leq g(b)|Q_R| +C\Big(R+\ell R^2\Big).$$
Let $\epsilon\in\R\setminus\{0\}$ such that $b+\epsilon\in(0,1)$. Using \eqref{eq:g'}, we get
$$F_{b+\epsilon,R}(\tilde f)\geq e_N(b+\epsilon,R)\geq g(b+\epsilon)|Q_R|-CR\,.$$
It is easy to notice that
\begin{equation}\label{eq:g(b+e)}
\begin{aligned}
\epsilon\int_{Q_R}|(\nabla-i\Ab_0) \tilde f|^2\,dy&=F_{b+\epsilon,R}(\tilde f)-F_{b,R}(\tilde f)\\
&\geq \Big(g(b+\epsilon)-g(b)\Big)|Q_R|-C\Big(R+\ell R^2\Big)\,.
\end{aligned}
\end{equation}
For $\epsilon>0$, we infer from  \eqref{eq:g(b+e)} the lower bound
$$\int_{Q_R}|(\nabla-i\Ab_0) \tilde f|^2\,dy\geq \frac{g(b+\epsilon)-g(b)}\epsilon |Q_R|-C\epsilon^{-1}\Big(R+\ell R^2\Big)\,.$$
Choosing $\epsilon=\max\Big(R^{-1/2},\ell^{1/2}\Big)$, we get further
$$\int_{Q_R}|(\nabla-i\Ab_0) \tilde f|^2\,dy\geq g'(b_+)|Q_R|-\mathrm r_1(R,\ell) |Q_R| \,,$$
where
$$\mathrm r_1(R,\ell)=C\Big(R^{-1/2}+\ell^{1/2}\Big)+\left|\frac{g(b+\epsilon)-g(b)}{\epsilon}-g'(b_+)\right|\to0~{\rm as~}~(R^{-1},\ell)\to 0\,.$$
In a similar fashion,  we choose $\epsilon=-\max\Big(R^{-1/2},\ell^{1/2}\Big)<0$ and infer from \eqref{eq:g(b+e)} the upper bound
$$ \int_{Q_R}|(\nabla-i\Ab_0) \tilde f|^2\,dy\leq g'(b_-)|Q_R|+\mathrm r_2(R)  |Q_R|\,,$$
where
$$ \mathrm r_2(R,\ell)=C\Big(R^{-1/2}+\ell^{1/2}\Big)+\left|\frac{g(b+\epsilon)-g(b)}{\epsilon}-g'(b_-)\right|
\to 0\quad~{\rm as~}~(R^{-1},\ell)\to 0\,.$$
To conclude, we choose $\mathrm r(R)=\max\Big(\mathrm r_1(R,\ell),\mathrm r_2(R,\ell)\Big)$.

\end{proof}

\begin{lem}\label{lem:L2}
There exists a function $\tilde{\mathrm r}(\cdot,\cdot)$ such that   $\lim_{(t^{-1},s)\to0}\tilde{\mathrm r}(t,s)=0$ and, under the assumptions in Lemma~\ref{lem:mag-grad}, the following inequality holds
$$b g'(b_+)-2g(b)-\tilde{\mathrm r}(R,\ell)\leq \frac1{|Q_R|}\int_{Q_R}|\tilde f(y)|^2\,dy\leq bg'(b_-)-2g(b)+\tilde{\mathrm r}(R,\ell)\,.$$
\end{lem}
\begin{proof}
By \eqref{rem:blow-up} and Proposition~\ref{prop:Ka-SIMA}, 
$$\Big|F_{b,Q_R}(\tilde f)-g(b)|Q_R| \Big|\leq  CR^{3/2}\,.$$
By the  formula for the $L^4$-norm of $\psi$ in Proposition~\ref{prop:Ka-SIMA} and a change of variables, we have
$$\Big|\int_{Q_R}|\tilde f (y)|^4\,dy+2g(b)|Q_R|\Big|\leq C\big( \ell+R^{-1} \big)|Q_R|\,.$$
Combining the aforementioned formulae and the one in Lemma~\ref{lem:mag-grad}, we get the formula for the integral of $|\tilde f|^2$.
\end{proof}

By rescaling, we deduce  from Lemma~\ref{lem:L2}:
\begin{thm}\label{thm:HK*}
Let $b\in(0,1)$. There exist $C,R_0,\kappa_0>0$ and a positive-valued function $\lambda(\cdot)$ such that
$\displaystyle\lim_{\kappa\to + \infty}\lambda(\kappa)=0$ and the following is true.

Suppose that
\begin{itemize}
\item $\kappa\geq\kappa_0$ and $H=b\kappa$\,;
\item $R_0\kappa^{-1}\leq \ell\leq \kappa_0^{-1}$\,;
\item $Q_\ell$  is the interior of a square of side length $\ell$ satisfying $\overline{Q_\ell}\subset\Omega$\,;
\item $(\psi,\Ab)_{\kappa,H}$ is a minimizer of the functional in \eqref{eq-3D-GLf}\,.
\end{itemize}
Then the following inequalities hold
$$bg'(b_+)-2g(b)-\lambda(\kappa)\leq \frac{1}{|Q_\ell|}\int_{Q_\ell}|\psi (x) |^2\,dx\leq bg'(b_-)-2g(b)+\lambda(\kappa)\,.$$
\end{thm}
Theorem~\ref{thm:HK*} improves the results in \cite{SS02}, where only  a non-optimal upper bound on the integral of $|\psi|^2$ is given (see \cite[Eq.~(1.18)]{SS02}). In Theorem~\ref{thm:HK*}, we not only  prove a lower bound on the integral of $|\psi|^2$, but also a matching upper bound in the case where $b\in\mathcal R$ (i.e. when $g'(b_+)=g'(b_-)$). 

~
\begin{proof}[Proof of Theorem~\ref{thm:HK}]~
The proof of the statements (2) and (3) regarding the estimate of the $L^2$-norm of $\psi$ and the weak convergence of $|\psi|^2$ both follow from Theorem~\ref{thm:HK*} in a standard manner, see e.g. \cite[Proof of Thm.~4.1]{Att}.

Now, the proof of statement (1) regarding the $L^2$-norm of the magnetic gradient is a consequence of statement (1) and the formulas in \eqref{eq:L4=en} and \eqref{eq:L4}.

The first inequality in statement (4) regarding the supercurrent results from statement (1) and the following inequality
$$|j(\psi,\Ab)|\leq |(\nabla-i\kappa H\Ab)\psi|\,,$$
which is a consequence of the definition of the supercurrent in \eqref{eq:sc} and the inequality in \eqref{eq:psi<1}.  

The 
other inequality for the $L^1$-norm of the supercurrent results from the inequality $$|j(\psi,\Ab)|\leq |\psi|  \,  |(\nabla-i\kappa H\Ab)\psi|\,,$$
the Cauchy-Schwarz inequality and the conclusions in Statements~(1) and (2).
\end{proof}

\section{New properties of the function $g$}\label{sec:g(b)}

\subsection{Universal estimates of $g(b)$}

As a by-product of the result in Theorem~\ref{thm:HK}, we get new properties of the function $g(\cdot)$ introduced in \eqref{eq:g}.


Using the classical bound $|\psi|\leq 1$ (see \eqref{eq:psi<1}), we deduce from \eqref{eq:L2} that
\begin{equation}\label{eq:L2a}
\forall~b\in(0,1)\,,\quad bg'(b_+)-2g(b)\leq 1\,.
\end{equation}
We can obtain an upper bound on the left-derivative of $g$ as well by expanding  the square in the inequality $\displaystyle\int_\Omega(1-|\psi(x)|^2)^2\,dx\geq0$ then using \eqref{eq:L4} and \eqref{eq:L2}:
\begin{equation}\label{eq:ub-g'-}
\forall~b\in(0,1)\,,\quad bg'(b_-)\leq \frac12+g(b)\,.
\end{equation}
Note that \eqref{eq:ub-g'-} is better than \eqref{eq:L2a} since $g'(b_+)\leq g'(b_-)\,,$ $g(b)\geq -\frac12$, hence $\frac12+g(b)\leq 1+2g(b)\,$.

\subsection{On the behavior of $g(b)$ as $b\to0_+$} 
Taking the limit as $b\to0_+$ in \eqref{eq:L2a} and noticing that $g'(b_\pm)\geq0$ and $g(0)=-\frac12\,$, we get
$$\lim_{b \to 0_+}g'(b_\pm)=0\,.$$
Consequently, there exists a sequence $(b_n)_{n\geq 1}\subset \mathcal R$ such that $b_n\to 0$ and $g'(b_n)\to0$ ($\mathcal R$ is defined in \eqref{eq:g'}).
On the other hand, it is proved in \cite{Ka-JFA} that as $b\to0_+$, 
\begin{equation}\label{eq:g(b)=}
g(b)=-\frac12+\frac{b}4\ln\frac1b+o\left(b\ln\frac1b\right)\,.\end{equation}
We deduce from this that:
\begin{itemize}
\item $g'(0_+)=+\infty$\,;
\item the function $b\mapsto g'(b_+)$ is not continuous at $0$\,;
\item The asymptotics in \eqref{eq:g(b)=} can not be differentiated,  i.e. the formula
$$g'(b)\sim\frac14\ln\frac1b-\frac14$$
does not hold as $b\underset{b\in\mathcal R}{\longrightarrow} 0_+$.   The aforementioned  sequence $(b_n)$ may violate this formula.
\end{itemize}

\subsection{The radial symmetry}
Next we try to extract more information about the function $g$ by exploiting the radial symmetry. The function $g$ may be expressed as follows
\begin{equation}\label{eq:g-disc}
\forall~b\in(0,1)\,,\quad g(b)=\lim_{R\to\infty} \frac{\mathfrak e_{\rm disc}(b,R)}{\pi R^2}\,,
\end{equation}
where
\begin{equation}\label{eq:en-disc}
\mathfrak e_{\rm disc}(b,R)=\inf\{ F_{b,D_R}(u)~:~u\in H^1_0(D_R)\}\,,
\end{equation}
$D_R=\{x\in\R^2~:~|x|<R\}$ and $F_{b,D_R}$ is the functional introduced in \eqref{eq:rGL}. The proof of \eqref{eq:g-disc} is standard (see \cite{AfSe, FK-cpde}). It follows by covering the disc $D(0,R)$ with squares $(Q_{R',j})_j$ with side-length $1\ll R'\ll R$ and using the estimates in \eqref{eq:g'} (for $r=R'$). We omit the technical details.

We restrict the functional $F_{b,D_R}(u)$ on configurations of the form
\begin{equation}\label{eq:fpures}
u(r,\theta)=e^{im\theta} f(r)\,,
\end{equation}
where $f:(0,R)\to\C$, $m\in\mathbb Z$ and $(r,\theta)$ denote the polar coordinates. 

Note that $u\in H^1_0((B(0,R))$ if and only if
$f\in\mathcal D_{m,R}$, where
\begin{equation}\label{eq:D-polar}
\mathcal D_{m,R}=\Big\{f~:~\sqrt{r}\,f',\sqrt{r}\,f,\frac{m}{\sqrt{r}}\,f\in L^2\big((0,R);\R\big)\,,~f(R)=0\Big\}\,.
\end{equation} 
Furthermore,
$$F_{b,D_R}(u)=G_{m,b,R}(f)\,,$$
where
\begin{equation}\label{eq:G-polar}
G_{m,b,R}(f)=2\pi\int_0^R \left(b|f'(r)|^2+b\Big(\frac{m}{r}-\frac{r}2\Big)^2|f(r)|^2-|f(r)|^2+\frac12|f(r)|^4\right)rdr\,.
\end{equation}
Consequently, we define the following ground state energy
\begin{equation}\label{eq:G-polar*}
\mathfrak e^{\rm 1D}(m,b,R)=\inf\{G_{b,m,R}(f)~:~f\in\mathcal D_{m,R}\}\,.
\end{equation}
A minimizer $f_{m,b,R}$ exists, can be selected real-valued and non-negative (because $|f_{m,b,R}|$ is a minimizer too) and  satisfies the following ODE 
\begin{equation}\label{eq:G-polar***}
-f''_{m,b,R}(r)-\frac1rf'(r)+\Big(\frac{m}{r}-\frac{r}2\Big)^2f_{m,b,R}(r)=\frac1b\big(1-|f_{m,b,R}(r)|^2\big)f_{m,b,R}(r)\quad{\rm in~}(0,R)\,.\end{equation}
 
 
When the magnetic field is absent (i.e. the term $\frac r2$ is dropped from \eqref{eq:G-polar***}) and  $R=+\infty\,$, \eqref{eq:G-polar***} has been studied in many papers, for example \cite{HH}.

Now we define
\begin{equation}\label{eq:gm(b)}
g_m(b)=\limsup_{R\to + \infty} \frac{\mathfrak e^{\rm 1D}(m,b,R)}{\pi R^2 }\,.
\end{equation}
We then have, 
\begin{equation}\label{eq:g<gm}
\forall~b\in(0,1)\,,\quad\forall~m\in\mathbb Z\,,\quad g(b)\leq g_m(b)\,.
\end{equation}

\begin{rem}
 A natural question is then to determine if for any $b\in (0,1)$ there exists $m\in \mathbb Z$ such that $g(b)=g_m(b)$ and if the discontinuity of $g'$ corresponds to the case when two $m$'s satisfy this property.
\end{rem}

\section{Extension to three dimensional domains}

The result in Theorem~\ref{thm:HK} can be easily extended to the three dimensional Ginzburg-Landau model. 
In this section,  $\Omega\subset\R^3$ denotes a bounded  smooth open  set with a smooth boundary. We introduce   the Ginzburg-Landau functional in $\Omega$ as follows \cite{FH-b, LP-3D},
\begin{multline}\label{eq-3D-GLf*}
\mathcal E^{\rm 3D}(\psi,\Ab)=\mathcal
E_{\kappa,H}^{\rm 3D}(\psi,\Ab)=
\int_\Omega\left[
|(\nabla-i\kappa
H\Ab)\psi|^2-\kappa^2|\psi|^2+\frac{\kappa^2}{2}|\psi|^4\right]\,dx\\
+\kappa^2H^2\int_{\R^3}|\curl\Ab-\beta|^2\,dx\,,
\end{multline}
where $\beta=(0,0,1)$. \\
The  configuration $(\psi,\Ab)$ belongs to the space $H^1(\Omega;\C)\times 
\dot{H}^1_{\Div,\Fb}(\R^3)$ with 
$\dot H^1_{\Div,\Fb}(\R^3)$  defined as follows. Let  $\dot H^1(\R^3)$ be the homogeneous Sobolev space, i.e. the closure
of $C_c^\infty(\R^3)$ under the norm $u\mapsto\|u\|_{\dot
  H^1(\R^3)}:=\|\nabla u\|_{L^2(\R^3)}$. Let further
  $\Fb(x)=(-x_2/2,x_1/2,0)$. Clearly $\Div \Fb=0$.
We define the space,
\begin{equation}\label{eq-3D-hs*}
\dot H^1_{\Div,\Fb}(\R^3)=\{\Ab~:~\Div \Ab=0\,,\quad~{\rm and}\quad
\Ab-\Fb\in \dot H^1(\R^3)\}\,.
\end{equation}
Now we define  the ground state energy,
\begin{equation}\label{eq-3D-gs*}
\E(\kappa,H)=\inf\big\{
\mathcal E^{\rm 3D}(\psi,\Ab)~:~(\psi,\Ab)\in H^1(\Omega;\C)\times \dot H^1_{\Div,\Fb}(\R^3)\big\}\,.
\end{equation}
This energy is estimated in \cite{FK-cpde} when $H=b\kappa$, $b\in(0,1)$ is a fixed constant and $\kappa\to\infty$. Using the methods in \cite{FK-cpde}, we may easily adapt the proof of Theorems~\ref{thm:HK} and \ref{thm:HK*} to get the following result:

\begin{thm}\label{thm:HK**}
For $b\in(0,1)$, there exist $C,R_0,\kappa_0>0$ and a positive-valued function $\lambda(\cdot)$ such that
$\displaystyle\lim_{\kappa\to + \infty}\lambda(\kappa)=0$ and the following is true.

Suppose that
\begin{itemize}
\item $\kappa\geq\kappa_0$ and $H=b\kappa$\,;
\item $R_0\kappa^{-1}\leq \ell\leq \kappa_0^{-1}$\,;
\item $Q_\ell$  is the interior of a cube of side length $\ell$ satisfying $\overline{Q_\ell}\subset\Omega$\,;
\item $(\psi,\Ab)_{\kappa,H}$ is a minimizer of the functional in \eqref{eq-3D-GLf*}\,.
\end{itemize}
Then the following inequalities hold
$$bg'(b_+)-2g(b)-\lambda(\kappa)\leq \frac{1}{|Q_\ell|}\int_{Q_\ell}|\psi|^2\,dx\leq bg'(b_-)-2g(b)+\lambda(\kappa)\,.$$
\end{thm}

As a consequence of Theorem~\ref{thm:HK**}, we can get that the minimizer $(\psi,\Ab)_{\kappa,H}$ satisfies the following weak convergence for $H=b\kappa$, $b\in\mathcal R$  and $\kappa\to\infty$\,:
$$|\psi|^2\to bg'(b)-2g(b)   \quad{\rm in~}\mathcal D'(\Omega)\,.$$

This result is complementary to the results in \cite{FKP-3D} and \cite{KN} devoted respectively to the regimes $b>1$ (surface superconductivity) and $b\to1_-$ (bulk superconductivity near $H_{C_2}$) for three dimensional superconducting samples.

\subsection*{Acknowledgments}
AK is supported by a grant from Lebanese university in the framework of `\'Equipe de Mod\'elisation, Analyse et Applications'. 



\begin{thebibliography}{100}

\bibitem{Ab} A. Abrikosov.
On the magnetic properties of superconductors of the second group. J. Exp. Theor. Phys. {\bf 5}
(1957), 1174-1182.

\bibitem{AfSe} A. Aftalion, S. Serfaty. Lowest Landau level approach in superconductivity for the Abrikosov lattice close to $HC2$.
{\it Selecta Math. (N.S.)} {\bf 13} (2007), 183-202.

\bibitem{Al} Y. Almog. Abrikosov lattices in finite domains. {\it Commun. Math. Phys.} {\bf 262} (2006),  677-702.

\bibitem{Al-s} Y. Almog. Non-linear surface superconductivity in the large $\kappa$ limit. {\it Rev. Math. Phys.} {\bf 16} (2004), 961-976. 

\bibitem{Att} K. Attar. The ground state energy of the two
dimensional Ginzburg-Landau functional with variable magnetic field.
 \newblock Ann. Inst. H. Poincaré Anal. Non Linéaire 32 (2015), no. 2, 325--345.

\bibitem{Att2} K. Attar. Energy and vorticity of the Ginzburg-Landau model with variable magnetic field.
\newblock  Asymptot. Anal. 93 (2015), no. 1-2, 75-114.

\bibitem{Att3} K. Attar. Pinning with a variable magnetic field for the Ginzburg-Landau model.
\newblock  Nonlinear Anal. 139 (2016), 1--54. 

\bibitem{CR} M. Correggi,  N. Rougerie. Boundary behavior of the Ginzburg-Landau order parameter in the surface superconductivity regime. {\it Arch. Rational Mech. Anal.} {\bf 219} (2015), 553-606.

\bibitem{F-sc} S. Fournais. On the semiclassical asymptotics of the current and magnetisation of a non-interacting electron gas at zero temperature in a strong constant magnetic field. {\it Ann. Henri Poincar\'e} {\bf 2} (2001), 1189-1212

\bibitem{FH-b} S. Fournais, B. Helffer. {\it Spectral Methods in
Surface Superconductivity.} Progress in Nonlinear Differential
Equations and Their Applications. {\bf 77} Birkh\"auser (2010).

\bibitem{FH-aif} S. Fournais, B. Helffer. Strong diamagnetism for general domains and applications.
{\it Annales de l'Institut Fourier} {\bf 57} (7) (2007), 2389-2400.

\bibitem{FK-am} S. Fournais, A. Kachmar. Nucleation of bulk superconductivity close to critical magnetic field. 
\newblock \newblock {\it Adv. Math.} {\bf 226},  1213-1258 (2011).

\bibitem{FK-cpde} S. Fournais, A. Kachmar. The ground state energy
of the three dimensional Ginzburg-Landau functional. Part~I. Bulk
regime.
\newblock  {\it Communications in Partial Differential Equations.} {\bf
38}, 339-383 (2013).

\bibitem{FKP-3D} S. Fournais, A. Kachmar, M. Persson. 
The ground state energy of the three dimensional Ginzburg-Landau functional. Part~II: Surface regime. {\it J. Math. Pures Appl.} {\bf 99}, 343-374 (2013).

\bibitem{dG} P.G. de\,Gennes. {\it Superconductivity of Metals and Alloys.}  Advanced Books Classics, Westview Press (1999).

\bibitem{HH} R-M. Herv\'e, M. Herv\'e.  \'Etude qualitative des solutions r\'eelles d'une \'equation diff\'erentielle li\'ee \`a l'\'equation de Ginzburg-Landau.
\newblock  {\it Ann. Inst. H. Poincar\'e Anal. Non Lin\'eaire} {\bf 11} (4), 427--440 (1994).

\bibitem{Ka-nf} A. Kachmar. A new formula for the energy of bulk superconductivity. \newblock  {\it Canad. Math. Bull.}  {\bf 59} (3), 553-563, (2016).


\bibitem{Ka-SIMA} A. Kachmar. The Ginzburg-Landau order parameter near the second critical field. 
\newblock {\it SIAM J. Math. Anal.}
 {\bf 46} (1), 572-587 (2014).

\bibitem{KN} A. Kachmar, M. Nassrallah. The distribution of $3$D superconductivity near the second critical field. 
{\it  Nonlinearity} {\bf 19} (9), 2856-2887 (2016).


\bibitem{Ka-JFA} A. Kachmar. The ground state energy of the three dimensional Ginzburg-Landau model in the mixed phase.  
\newblock {\it J. Funct. Anal.} {\bf 261}, 3328-3344 (2011).

\bibitem{LP-3D} K. Lu, X.-B. Pan. Surface nucleation of superconductivity in 3-dimensions.
{\it J. Differential Equations} {\bf 168} (2),  386-452 (2000).

\bibitem{P-cmp} X.B. Pan.  Surface superconductivity in applied magnetic fields above $H_{c2}$. {\it Comm. Math. Phys.}  {\bf 228}, (2002), 327-370.

\bibitem{SS-b}  E. Sandier, S. Serfaty. {\it Vortices for the Magnetic
Ginzburg-Landau Model.} Progress in Nonlinear Differential Equations
and their Applications  {\bf 70},  Birkh\"auser  (2007).

\bibitem{SS-cmp} E. Sandier, S. Serfaty. 
From the Ginzburg-Landau Model to Vortex Lattice Problems. {\it Comm. Math. Phys.}
{\bf 313} (3) (2012),  635-743.


\bibitem{SS02} E. Sandier, S. Serfaty. The decrease of bulk
superconductivity close to the second critical field in the
Ginzburg-Landau model. {\it SIAM. J. Math. Anal.} {\bf 34} No. 4
(2003), 939--956.

\bibitem{SS-rmp} E. Sandier, S. Serfaty. 
On the energy of type-II superconductors in the mixed phase. {\it Rev. Math. Phys.} {\bf 12} (9) (2000), 
1219-1257.

\end{thebibliography}
\end{document}